\documentclass[12pt]{article}

\usepackage{amsmath}
\usepackage{amsfonts}
\usepackage{amssymb}
\usepackage{amscd}
\usepackage{amsthm}
\usepackage{bbm}
\usepackage{color}
\usepackage[linktocpage=true,plainpages=false,pdfpagelabels=false]{hyperref}

\input xypic

\usepackage[matrix,arrow,curve]{xy}

\usepackage{color}



\newcommand{\quash}[1]{}

\newtheorem{Th}{Theorem}[section]

\newtheorem{prop}[Th]{Proposition}
\newtheorem{nt}[Th]{Remark}
\newtheorem{lemma}[Th]{Lemma}

\newtheorem{example}[Th]{Example}

\newfont{\ssdbl}{msbm8}
\newfont{\sdbl}{msbm9}
\newfont{\dbl}{msbm10 at 12pt}
\newcommand{\eqdef}{\stackrel{\rm def}{=}}

\newcommand{\oo}{{\cal O}}
\newcommand{\ff}{{\cal F}}

\newcommand{\g}{{\cal G}}
\newcommand{\ad}{{\cal A}}

\newcommand{\Spec}{\mathop {\rm Spec}}
\newcommand{\Supp}{\mathop {\rm supp}}

\newcommand{\Frac}{\mathop {\rm Frac}}

\newcommand{\da}{\mathbb{A}}
\newcommand{\dz}{\mathbb{Z}}
\newcommand{\dc}{\mathbb{C}}
\newcommand{\dr}{\mathbb{R}}
\newcommand{\dt}{\mathbb{T}}

\newcommand{\Ker}{{\rm Ker}\:}
\newcommand{\Image}{{\rm Im}\:}

\newcommand{\lrto}{\longrightarrow}

\def\Q{{\mathbb Q}}

\newcommand{\nsubset}{\not{\subset}}

\begin{document}

\author{
D. V. Osipov\footnote{The author is partially supported by Laboratory of Mirror Symmetry NRU HSE, RF Government grant, ag. № 14.641.31.0001}
}

\title{Arithmetic surfaces and adelic quotient groups
}
\date{}

\maketitle

\quad \qquad \qquad
{\em Dedicated to A.~N.~Parshin on the occasion of his 75th birthday}

\begin{abstract}
We explicitly calculate an arithmetic adelic quotient  group for a locally free sheaf on an arithmetic surface
when the fiber over the infinite point of the base is taken into account. The calculations are presented via a short exact sequence.
We relate the last term of this short exact sequence with the
projective limit of groups which are finite direct products of copies of
 one-dimensional real torus and  are connected with
first cohomology groups of locally free sheaves on the arithmetic surface.
\end{abstract}

\section{Introduction}

The ring of adeles for number fields and algebraic curves was introduced by C.~Chevalley and A.~Weil.

Higher adeles (or ring of adeles in higher dimensions) were introduced by A.~N.~Parshin and A.~A.~Beilinson.
A.~N.~Parshin defined adeles for smooth algebraic surfaces over a field in~\cite{P}.
A.~A.~Beilinson defined adeles for arbitrary Noetherian schemes in a short note~\cite{Be}.
Later proofs of Beilinson results concerning adelic resolutions of quasicoherent sheaves appeared in the paper of A.~Huber~\cite{H}.

The ultimate goal of the higher adeles program  is the generalization of the Tate-Iwasawa method from one-dimensional case to the case of higher dimensions, see~\cite{P1,P2}.
The Tate-Iwasawa method allows to obtain a meromorphic continuation to all of $\mathbb{C}$ and a functional equation for zeta-
and $L$-functions of number fields and the fields of rational functions of curves defined over finite fields, and this method works simultaneously
 in the number theory case and in the geometric case, see~\cite{T}.

In arithmetic algebraic  geometry, a well-known approach can be used where the problem is solved first for function fields by a method that can be transferred to the scheme part  of an arithmetic surface  related to a number field.  The next final step is to include  the Archimedean fibers of the surface.  This approach was quite successful for  Faltings's proof of  the Mordell conjecture, see~\cite{PZ}. The consideration of the Archimedean fibers was based on the Arakelov theory of arithmetic surfaces, see~\cite{Ar}.

Recently,  A.~N.~Parshin has developed a new version of the  Tate-Iwasawa method, see~\cite{P3}, where he has  removed the well-known manipulations with formulas  and  replaced  them by functoriality and duality considerations. The parts of these constructions can be done for the case of an algebraic surface over a finite field \footnote{This was shown in  his talks at the Steklov Mathematical Institute in November 2017.}.
In this paper, we extend the adelic results that are known for the algebraic surfaces to the case of  arithmetic surfaces.

If we consider an arithmetic surface~$X$, i.e. a two-dimensional normal integral scheme which is surjectively fibred over     $\mathop{\rm Spec} {\mathbb Z}$,
then the Parshin-Beilinson ring of adeles ${\mathbb A}_X$ of $X$ does not take into account at all, what is in the fiber  $X \times_{\mathbb Z} {\mathbb R}$
over ``the infinite point'' of $\mathop{\rm Spec} {\mathbb Z}$,
i.e. over the Archimedean valuation of the ring $\mathbb Z$.
A.~N.~Parshin and the author of this paper defined in~\cite{OsipPar2} a ring of arithmetic adeles ${\mathbb A}_X^{\rm ar}$ of the arithmetic surface $X$,
which takes into account the fiber of $X$ over ``the infinite point'' of $\mathop{\rm Spec} {\mathbb Z}$.
 Informally speaking,
the ring ${\mathbb A}_X$ is a complicated restricted product of two-dimensional local fields over all pairs: a closed point in $X$ and a formal branch
at $x$ of an integral $1$-dimensional subscheme  of $X$, where a two-dimensional local field is a finite extension of a field $\Q_p((t))$
or a field $\Q_p\{\{ t\}\}$, see more details, e.g., in~\cite[\S~2.1]{Osi}. Now,
to obtain the ring ${\mathbb A}_X^{\rm ar}$  from the Parshin-Beilinson ring of adeles $\da_X$, we add the fields $\dr((t))$
or $\dc((t))$ associated with ``horizontal curves'' on $X$ and ``infinite points'' of ``horizontal curves''. In other words, we add the restricted product of completions of local fields for points with non-transcendental coordinates on $X \times_{\mathbb Z} {\mathbb R}$.
It is important, as in the case of Parshin-Beilinson adeles of a scheme, to define various subgroups of ${\mathbb A}_X^{\rm ar}$  so that the results obtained, have analogy as with the classical one-dimensional case of number  fields as well as with the case of projective surfaces over a field.

The goal of this paper is to calculate explicitly an adelic quotient group:
\begin{equation}  \label{ad_quot}
\da_{X}^{\rm ar} (\ff) / \left( \da_{X, 01}^{\rm ar}(\ff) + \da_{X, 02}^{\rm ar}(\ff) \right)  \, \mbox{,}
\end{equation}
where $\ff$ is a locally free sheaf on an arithmetic surface $X$, and $\da_{X}^{\rm ar}(\ff)$,  $\da_{X, 01}^{\rm ar}(\ff)$,  $\da_{X, 02}^{\rm ar}(\ff)$
are the group of arithmetic higher adeles of $X$ and its subgroups, see more details below, in Section~\ref{Statement}.
The subgroups  $\da_{X, 01}^{\rm ar}(\ff)$  and  $\da_{X, 02}^{\rm ar}(\ff)$ were introduced in papers~\cite{SW1,SW2}. (Close subgroups were considered also in~\cite{Zh}.)

The subgroup $\da_{X, 01}^{\rm ar}(\ff) + \da_{X, 02}^{\rm ar}(\ff)$ is an analog of the subgroup $K$ in the adelic ring $\da_K$, where
$K$ is a global field, i.e. $K$ is either a number field or the field of rational functions of algebraic curve over a finite field.
{(For example, if $K$ is the field of rational functions of an algebraic projective curve $Z$ over a finite field, then $K=\da_0(\oo_Z)$ and
$\da_K = \da_{01} (\oo_Z)$ in the adelic complex for the structure sheaf $\oo_Z$ of the curve $Z$, see, e.g.,~\cite[\S3]{Osi}.)}
Then as it is well-known that for the global field $K$ the group
$\da_K/K$ is compact, and this statement has a lot of applications. {By explicit calculations we obtain similar result for the case of arithmetic surface $X$ and the group $\da^{\rm ar}_X(\ff)$.}

We recall explicit calculation when  $K$ is a number field such that $[K : \Q]=l$ and $E$ is the ring of integers in $K$.  The strong approximation theorem
immediately
 implies an exact sequence
 \begin{equation}  \label{one-dim}
 0 \lrto \prod_{\sigma} \hat{E}_{\sigma}  \lrto \da_K/ K \lrto \left( \prod_{\upsilon} K_{\upsilon} \right)/ E  \lrto 0  \, \mbox{,}
 \end{equation}
 where $\da_K$ is the ring of adeles of the field $K$, $\sigma$ runs over the set of maximal ideals of the ring $E$, $\hat{E}_{\sigma}$ is the corresponding completion, $\upsilon$ runs over the (finite) set of Archimedean places of the field $K$, $K_{\upsilon}$ is the completion. We note that the last non-zero term  in exact sequence~\eqref{one-dim} is isomorphic to $\mathbb{T}^{\, l}$, where ${\mathbb{T}} = \dr/\dz$ is the torus.

 On the other hand, for a normal projective irreducible algebraic surface $Y$ over a field $k$ we fix an  ample  divisor $\tilde{C}$ on $Y$.
   Then the complement $U = Y \setminus \Supp \tilde{C}$ is affine. And this situation can be considered as an analog of the $1$-dimensional case,
   but we fix instead of finite number of points of a global field (Archimedean places in number theory case) a divisor with an affine complement.

   For a locally free sheaf $\ff$ on $Y$, the corresponding adelic quotient group  was calculated in~\cite{Osip1}  and~\cite[\S~14]{OsipPar2} in the following form:
 \begin{multline}  \label{seq_th}
0 \lrto
G_1  \lrto
\da_Y(\ff)/ \left(  \da_{Y, 01}(\ff) + \da_{Y, 02}(\ff)   \right)   \lrto  G_2 \lrto 0
\end{multline}
with the linearly compact $k$-vector space\footnote{Sometimes we will use notation for the quotient group $N_1/N_2$ as ``fraction'' $\frac{N_1}{N_2}$.}
$$
G_1  \; \simeq  \;
 \frac{ \prod\limits_{D \subset Y, D \nsubset \tilde{C} } \left( \left(\frac{\mathop{{\prod}'}\limits_{x \in D}  \oo_{K_{x,D}}}{ \hat{\oo}_D}  \right) \otimes_{\hat{\oo}_D} \hat{\ff}_D \right)}{ \prod\limits_{x \in U} \hat{\ff}_x  }
$$
where $x$ runs over the set of (closed) points in  the  subscheme $U $, $D$ runs over the set of irreducible curves on $Y$ such that $D \, {\not{\hspace{-0.1cm}\subset}} \Supp \tilde{C}$,  ${\prod}'$ means an adelic (on algebraic surface) product, the ring $\oo_{K_{x,D}}$
is the product of discrete valuation rings from the finite set of two-dimensional local fields constructed by a pair $x \in D$ (for example,
if $K_{x,D} =k((u))((t))$ is a two-dimensional local field, then $\oo_{K_{x,D}}= k((u))[[t]]$), $\hat{\ff}_x$ is the completion of the stalk of $\ff$ at $x$, $\hat{\ff}_D$ is the completion of the stalk $\ff$ at the generic point of $D$.

The $k$-vector space $G_2$ is also linearly compact. We recall (see~\cite[\S~14.3]{OsipPar2}) that in the simplest case $Y = C \times_k C$,
$C =  \mathbb{P}^1_{k}$,
 $\tilde{C} = C \times y + y \times C$ ($y$ is a fixed  $k$-rational point on $C$),
 and $\ff = \oo_Y$ we have
$$
G_2\simeq  \frac{ k((u))((t))}{  k[u^{-1}]((t))  + k((u))[t^{-1}]  }  \, \mbox{.}
$$

Now if we change the locally linearly compact field $k((u))$ to the locally compact field $\dr$, and the discrete subspace $k[u^{-1}]$
to the discrete subgroup $\dz$, then $G_2$ should be equal to $t \dr[[t]]/ t \dz[[t]] \simeq t {\mathbb{T}}[[t]] $.

We give in this paper an explicit calculation
 for   the adelic  quotient group~\eqref{ad_quot} on an arithmetic surface $X$ and a locally free sheaf $\ff$ on $X$.
  This is the content of Theorem~\ref{main-Th} in section~\ref{Statement} of the paper. We don't formulate this theorem here, in introduction, but note that the result will generalize the result from formula~\eqref{one-dim} for a number field case,  and this result will be analogous to the result~\eqref{seq_th} for a projective algebraic surface $Y$ case.

More exactly, the answer, which is obtained in Theorem~\ref{main-Th}, gives a presentation of the group~\eqref{ad_quot}
via a short exact sequence,  where the first term is written analogously to the first term in exact sequence~\eqref{seq_th} (but we have to take into account additional terms like $\dr((t))$ or $\dc((t))$) and the group in this term is compact as well as $G_1$ is a linearly compact $k$-vector space. Moreover, the last term of this exact sequence in the simplest case  $X =\mathbb{P}^1_{\dz}$ and $\ff = \oo_X$
equals to $t \dr[[t]]/ t \dz[[t]] \simeq t {\mathbb{T}}[[t]] $ (this follows from explicit calculations given for $\mathbb{P}^1_{\dz}$ in
  Example~\ref{Examp}).

 We give another interesting remark.
 In Proposition~\ref{theta} we prove that  the last term in the short exact sequence for the group~\eqref{ad_quot}
 is isomorphic to the projective limit over $n \le 0$ of groups $H^1(X, \ff(nC)) \otimes_{\dz} {\mathbb T}$,
 where $C$ is an ample (fixed before calculations) Cartier divisor on $X$.
  Each of this groups is a finite direct product of copies of
  $\mathbb T$, and the projective limit is isomorphic to a countable product of groups isomorphic to ${\mathbb T}$.
 (This is evident on the Pontryagin dual side.)
 On the other hand, going back to exact sequence~\eqref{one-dim}, we note that in one-dimensional Arakelov geometry there is the Riemann-Roch theorem in the form of Poisson summation formula. It uses theta-functions constructed by the right hand side of formula~\eqref{one-dim}, i.e. by the cocompact lattice $E$ in the finite-dimensional $\dr$-vector space $\prod\limits_{\upsilon} K_{\upsilon}$, see more details in Remark~\ref{Arakelov}  below.

This paper is organized as follows. In Section~\ref{Statement}
we recall the basic notions from Parshin-Beilinson adelic theory and give necessary definitions for groups of arithmetic adeles on an arithmetic surface. Besides, we formulate here the main result of this paper, Theorem~\ref{main-Th}.
In Section~\ref{proof_th} we prove this theorem. Firstly, in Section~\ref{first_term} we calculate the first term of the corresponding short exact sequence. Secondly, in Section~\ref{last_non-zero_term} we calculate the last term of the corresponding short exact sequence.
Finally,  in Section~\ref{appl}  we prove that the last term of the corresponding short exact sequence is isomorphic to the projective limit of
 groups which are finite direct products of copies of
  ${\mathbb T}$ and these groups are related with the first cohomology groups of sheafs $\ff(nC)$.

\section{Statement of the result} \label{Statement}

Let $X$ be a two-dimensional integral normal scheme. We suppose that the natural morphism $f : X \to \Spec \dz$ is projective and surjective.
Then it follows that $f$ is a flat morphism.
We will call $X$ as an arithmetic surface.

Let $C = \bigcup\limits_{1 \le i \le w} C_i$ be a reduced one-dimensional subscheme of $X$ such that for any $1 \le i \le w$ the one-dimensional subscheme $C_i$ is integral and $f(C_i)= \Spec \dz$.

We suppose that an open subscheme $U = X \setminus C  \stackrel{j}{\lrto} X$ is an affine scheme. We note that this condition on $U$ is satisfied, for example, if a coherent sheaf $\oo_X(nC)$ is an ample invertible sheaf on $X$ for some integer $n \ge 1$.
 By~\cite[Chapter~5, Corollary~3.24]{Liu}   an invertible  sheaf $\mathcal L$ on $X$ is ample if and only if an invertible sheaf $(i_p)^* {\mathcal L}$ on a subscheme $ f^{-1}(p) \stackrel{i_p}{\longrightarrow} X$  is ample for any closed point $p$ of $\Spec  \dz$.

\bigskip

Let $\ff$ be a locally free sheaf of $\oo_X$-modules on $X$. For any point $q$ in $X$, let $\hat{\ff}_q$ be the completion of the stalk $\ff_q$ of the sheaf $\ff$ at a point $q$ of $X$.

By $\da_X(\ff)$ we denote the Parshin-Beilinson adelic group of $\ff$ on $X$:
$$
\da_{X}(\ff) = \da_{X, 012}(\ff)= \mathop{{\prod}'}_{x \in D} K_{x,D}(\ff)  \, \subset \, \mathop{{\prod}}_{x \in D} K_{x,D}(\ff) \, \mbox{,}
$$
where $x \in D$ is a pair with a closed point $x \in X$ and an integral closed one-dimensional subscheme ${D \subset X}$\footnote{Follow notation that comes from the higher adeles on an algebraic surface we will use sign ``$\in$'' when  $ x$ is a closed point in the corresponding scheme, and the sign ``$\subset$'' when a one-dimensional integral scheme $D$ is a closed subscheme in  $X$.},
the ring $K_{x,D} = \prod\limits_{1 \le i \le l} K_i$ with a two-dimensional local field $K_i$ being the completion of the field $\Frac \hat{\oo}_x$ with respect to a height prime ideal of the ring $\hat{\oo}_x$ associated with the ideal of the subscheme $D$ restricted to $\Spec \hat{\oo}_{x}$,
the group $K_{x,D}(\ff)  = \hat{\ff}_x \otimes_{\hat{\oo}_x} K_{x,D} $ (see more details in~\cite{H}, \cite{Osi}, \cite[\S~3.1]{Osip1}).

We denote $\oo_{K_{x,D}}(\ff) =  \hat{\ff}_x \otimes_{\hat{\oo}_x} \oo_{K_{x,D}}$, where the ring $\oo_{K_{x,D}}= \prod\limits_{1 \le i \le l} \oo_{K_i}$ with $\oo_{K_i}$
being the discrete valuation ring in a two-dimensional local field $K_i$ associated with the pair $x \in D$.

By $\prod'$ we will denote the adelic product on $X$, i.e. the intersection with $\da_X(\ff)$ inside of $\prod$.

We recall that there are the natural subgroups $\da_{X, 01}(\ff)$ and $\da_{X, 02}(\ff)$ of the adelic group $\da_{X}(\ff)$:
\begin{gather}  \label{subgr}
\da_{X,01}(\ff)= \mathop{{\prod}'}_{D} K_{D}(\ff)  = \da_X(\ff)  \cap \prod_D K_D(\ff)   \, \mbox{,}  \\
\label{subgr2}
\da_{X, 02}(\ff) = \mathop{{\prod}'}_{x} K_{x}(\ff) = \da_X(\ff)  \cap \prod_x K_x(\ff)  \, \mbox{,}
\end{gather}
where the intersections are taken inside the group $\prod\limits_{x \in D} K_{x,D}(\ff)$, and
the group\footnote{By $\hat{\ff}_D $
we mean the completion of the stalk of $\ff$ at the generic point of $D$, i.e. at the non-closed point on $X$ whose closure coincides with $D$.} \linebreak ${K_D(\ff) = \hat{\ff}_{D}  \otimes_{\hat{\oo}_D} K_D}$, the field $K_D$ is the completion of the field $\Q(X)$ of rational functions on $X$ with respect to discrete valuation given by $D$, the group $K_x(\ff) = \hat{\ff}_x  \otimes_{\hat{\oo}_x} K_x$, the ring $K_x$ is the localization of the ring $\hat{\oo}_x$
with respect to the multiplicative system $\oo_x \setminus 0$. The products $\mathop{\prod}\limits_D$ and $\mathop{\prod}\limits_x$
are diagonally embedded into the product $\mathop{\prod}\limits_{x \in D}$ via the natural maps:
$$
\xymatrix{
p_{x,D, \ff} \; : \; K_x(\ff) \; \ar@{^{(}->}[r]   & \;  K_{x,D}(\ff)
}
\qquad \mbox{and} \qquad
\xymatrix{
q_{x,D, \ff} \; : \; K_D(\ff)  \; \ar@{^{(}->}[r]  & \; K_{x,D}(\ff)  \,  \mbox{.}
}
$$

\bigskip

Let $V$ be a locally linearly compact vector space over a field $k$. Let $k' / k$ be a field extension. We denote
$$
V \hat{\otimes}_k k' \; \eqdef \; \varprojlim_{W} \left( (V/W) \otimes_k k' \right)  \; = \; \varinjlim_{W'}  \varprojlim_{W \subset W'} \left( (W'/W) \otimes_k k' \right)   \, \mbox{,}
$$
where $W$ and $W'$ run over the set of open linearly compact $k$-vector subspaces of $V$.
For example, $\Q((t)) \hat{\otimes}_{\Q} \dr = \dr((t)) $.

Let a curve $X_{\Q}= X \times_{\Spec \dz} \Spec \Q$. Let $\da_{X_{\Q}}(\ff_{\Q})$ be the adelic group of the locally free sheaf $\ff_{\Q} = (i_{\Q})^* \ff$ on the generic fibre   $X_{\Q} \stackrel{i_{\Q}}{\lrto} X  $. Then $\da_{X_{\Q}}(\ff_{\Q})$ is a locally linearly compact $\Q$-vector space
with the base of open subspaces given by divisors on the curve $X_{\Q}$.

The arithmetic adelic group of $X$ was introduced in~\cite[Example~11]{OsipPar2} (see also applications in~\cite[\S~4]{Osi2016}). By definition,
\begin{equation}  \label{ar_adel}
\da_{X}^{\rm ar}(\ff) \; \eqdef \; \da_X(\ff) \oplus \left(  \da_{X_{\Q}}(\ff_{\Q}) \hat{\otimes}_{\Q} \dr  \right)  \, \mbox{.}
\end{equation}
We note that
$$
\da_{X_{\Q}}(\ff_{\Q}) \hat{\otimes}_{\Q} \dr \simeq \mathop{{\prod}'}_{D, \, f(D)= \Spec \dz} \left( K_{D}(\ff) \hat{\otimes}_{\Q} \dr \right) \; \mbox{,}
$$
where $\mathop{{\prod}'}$ means here the restricted (as in adeles on a curve) product with respect to the subgroups $\hat{\ff}_D \hat{\otimes}_{\Q} \dr $.

\medskip

The subgroups $\da_{X, 01}^{\rm ar}(\ff)$ and $\da_{X,02}^{\rm ar}(\ff)$ of the group $\da_{X}^{\rm ar}(\ff)$ were introduced in~\cite{SW1,SW2}.
The subgroup $\da_{X, 01}^{\rm ar}(\ff)$
coincides with $\da_{X, 01}(\ff)$ but is mapped into both summands of $\da_{X}^{\rm ar}(\ff)$ of formula~\eqref{ar_adel}, where the first map is the natural embedding from the definition in formula~\eqref{subgr}, and the second map is the composition of the natural maps:
\begin{equation}  \label{second_embed}
\da_{X, 01}^{\rm ar}(\ff)  \lrto \mathop{{\prod}'}_{D, \, f(D)= \Spec \dz} K_{D}(\ff)  \; \hookrightarrow \mathop{{\prod}'}_{D, \, f(D)= \Spec \dz} \left( K_{D}(\ff) \hat{\otimes}_{\Q} \dr \right) \; \simeq \; \da_{X_{\Q}}(\ff_{\Q}) \hat{\otimes}_{\Q} \dr  \,  \mbox{.}
\end{equation}
The subgroup $\da_{X, 02}^{\rm ar}(\ff)$ is, by definition,
\begin{equation}  \label{ad-ar-02}
\da_{X, 02}^{\rm ar}(\ff)  \eqdef \da_{X, 02}(\ff)  \oplus \left( \ff_{\eta} \otimes_{\Q}  \dr \right)  \, \mbox{,}
\end{equation}
where $\eta$ is the generic point of $X$, the first summand in~\eqref{ad-ar-02} is naturally embedded into the first summand of~\eqref{ar_adel}
by formula~\eqref{subgr2},
and the second summand in~\eqref{ad-ar-02} is naturally diagonally embedded into the second summand of~\eqref{ar_adel}.

\bigskip

We recall the definition of a subgroup $A_C(\ff)  \subset \prod\limits_{1 \le i \le w} K_{C_i}(\ff)$ from~\cite[\S~3.2.2]{Osip1}.

For this goal we first define  the subgroup $B_{x,C}(\ff) \subset K_x(\ff)$  as
\begin{equation}   \label{B-form}
B_{x,C}(\ff) \eqdef \bigcap_{D \ni x, D \nsubset C} p_{x,D, \ff}^{-1} \left( p_{x,D, \ff}(K_x(\ff))  \cap \oo_{K_{x,D}}(\ff) \right)  \, \mbox{.}
\end{equation}
By remark after Lemma~1 from~\cite[\S~3.2.1]{Osip1}  we also have
$$
B_{x,C}(\ff) \simeq \hat{\ff}_x   \otimes_{\oo_x} (j_* j^* \oo_X)_x  \mbox{.}
$$

Now the  subgroup $A_C(\ff) $ is defined as the image of the  projection of the group
$\Ker \Xi$ to the group   $ \prod\limits_{1 \le i \le w}  K_{C_i}(\ff)$, where the map
\begin{equation}  \label{Xi}
\Xi \, : \,  \prod\limits_{1 \le i \le w} K_{C_i}(\ff)  \oplus \prod_{x \in C} B_{x,C}(\ff)  \lrto   \prod\limits_{1 \le i \le w}
\prod_{x \in C_i}  K_{x,C_i}(\ff)       \, \mbox{,}
\end{equation}
and   $\Xi (z \oplus v) = \prod\limits_{1 \le i \le w}
\prod\limits_{x \in C_i} q_{x, C_i, \ff}(z)   -  \prod\limits_{1 \le i \le w}
\prod\limits_{x \in C_i} p_{x, C_i, \ff}(v) $ for elements $z \in \prod\limits_{1 \le i \le w} K_{C_i}(\ff) $
and $v \in \prod\limits_{x \in C} B_{x,C}(\ff) $.
By~\cite[Proposition~1]{Osip1} we also have
$$
A_C(\ff)  \simeq \varinjlim\limits _{n} \varprojlim\limits_{m < n}  H^0 \left( X, \ff \otimes_{\oo_X} \left( \oo_X(nC)/   \oo_X(mC) \right)  \right)
\, \mbox{.}
$$

Let an affine curve $U_{\Q} = U \times_{\Spec \dz} \Spec \Q$.
We denote a $\Q$-vector space
$$A_{U_{\Q}}(\ff)  \eqdef H^0 (U_{\Q}, \ff_{\Q} \mid_{U_{\Q}})  \, \mbox{.}$$

The main result of this paper is the following theorem.
\begin{Th}  \label{main-Th}
Let the group $\Phi = \da_{X}^{\rm ar} (\ff) / \left( \da_{X, 01}^{\rm ar}(\ff) + \da_{X, 02}^{\rm ar}(\ff) \right)$. Then there is an exact sequence:
\begin{multline}
0 \lrto
\frac{
\left( \mathop{{\prod}'}\limits_{x \in D, D \nsubset C} \oo_{K_{x,D}}(\ff) \right)
\oplus
\left(
\prod\limits_{D \nsubset C, \, f(D)= \Spec \dz} \left( \hat{\ff}_D \hat{\otimes}_{\Q} \dr   \right)
\right)
}{
\prod\limits_{D \subset X, D \nsubset C}  \hat{\ff}_D
+ \prod\limits_{x \in U} \hat{\ff}_x
}
\lrto \Phi
\lrto
 \\ \label{main-seq} \lrto
\frac{\prod\limits_{1 \le i \le w} \left( K_{C_i}(\ff) \hat{\otimes}_{\Q} \dr \right) }{A_C(\ff) + \left( A_{U_{\Q}}(\ff) \otimes_{\Q}  \dr  \right) }
\lrto
0 \, \mbox{,}
\end{multline}
where the group $\prod\limits_{D \subset X, D \nsubset C}  \hat{\ff}_D$ is canonically mapped to the both summands of the ``numerator'' (cf.~formula~\eqref{second_embed}), and
the group  $\prod\limits_{x \in U} \hat{\ff}_x$ is mapped only to the first summand of the ``numerator''.
\end{Th}

\begin{nt}{\em
As it was shown in~\cite[\S~14.3, Remark~27]{OsipPar2}, the first non-zero term in exact sequence~\eqref{main-seq} is a compact group.
The last non-zero term in exact sequence~\eqref{main-seq} will be analyzed in Section~\ref{appl}  below.
}
\end{nt}
\begin{example}  \label{Examp} {\em
In the simplest case when $X = \mathbb{P}^1_{\dz}$, $C$ is a hyperplane section, $U = \mathbb{A}^1_{\dz}$,  $\ff = \oo_X$ we have
the following subgroups which appear in the last non-zero term  in exact sequence~\eqref{main-seq}:
$$
\prod\limits_{1 \le i \le w} \left( K_{C_i}(\ff) \hat{\otimes}_{\Q} \dr \right) \simeq \dr((t)) \;  \mbox{,}  \qquad
A_C(\ff) \simeq \dz((t))
 \;  \mbox{,}  \qquad
 A_{U_{\Q}}(\ff)  \otimes_{\Q} \dr \simeq \dr[t^{-1}]  \, \mbox{.}
$$
}
\end{example}

\section{Proof of the theorem}  \label{proof_th}
In this section we prove Theorem~\ref{main-Th}.
In section~\ref{first_term} we will calculate the first non-zero term in exact sequence~\eqref{main-seq}.
In section~\ref{last_non-zero_term}  we will calculate the last non-zero term in exact sequence~\eqref{main-seq}.

We note also that we will use adelic complex ${\mathcal A}_Y(\g)$ on a two-dimensional integral normal finite type over $\Spec \dz$ scheme $Y$ for a locally free sheaf $\g$
of $\oo_Y$-modules:
$$
\begin{array}{ccccc}
\da_{Y, 0}(\g) \oplus \da_{Y,1}(\g) \oplus \da_{Y,2}(\g) & \lrto    & \da_{Y,01}(\g) \oplus \da_{Y,02}(\g) \oplus \da_{Y,12}(\g) &
\lrto & \da_{Y,012}(\g) \\
(a_0, a_1, a_2)            & \mapsto & (a_1 - a_0, a_2 - a_0, a_2 -a_1)  &
& \\
                           &        & (a_{01}, a_{02}, a_{12})             &
			   \mapsto & a_{01} - a_{02} + a_{12}  \mbox{.}
\end{array}
$$
Here subgroups $\da_{Y,01}(\g)$ and $\da_{Y,02}(\g)$
 of the group $\da_{Y,012}(\g)=\da_Y(\g)$ were defined before, see formulas~\eqref{subgr}-\eqref{subgr2}.
 By definition, the other subgroups of the group $\da_Y(\g)$ used in the adelic complex are:
 \begin{equation}  \label{ad-subgr}
 \da_{Y,12}(\g)=  \mathop{{\prod}'}\limits_{x \in D, D \nsubset C} \oo_{K_{x,D}}(\ff) \, \mbox{,} \quad
 \da_{Y, 0}(\g)= \g_{\eta}  \, \mbox{,} \quad \da_{Y,1}(\g)= \prod_{D \subset Y} \hat{\g}_D \, \mbox{,}  \quad
 \da_{Y,2}(\g) = \prod_{x \in Y} \hat{\g}_x   \, \mbox{,}
\end{equation}
where $\eta$ is the generic point of the scheme $Y$, and pairs $x \in D$ on $Y $ are as in Section~\ref{Statement}.

The main property of the adelic complex is that $H^i(\ad_Y(\g))= H^i(Y, \g)$ for $0 \le i \le 2$.
In particular, if $H^i(Y, \g) =0$ for $i =1$ or $i=2$, then the condition that $i$th cocycle is a coboundary in the adelic complex gives analogs of approximations theorems  for two-dimensional schemes (compare with~\cite[\S~2]{Osip1}).

\subsection{First non-zero term in exact sequence~\eqref{main-seq}}
In this section we calculate the first non-zero term in exact sequence~\eqref{main-seq}.

\label{first_term}
We consider a natural embedding:
$$
\left( \mathop{{\prod}'}\limits_{x \in D, D \nsubset C} \oo_{K_{x,D}}(\ff) \right)
\oplus
\left(
\prod\limits_{D \nsubset C, \, f(D)= \Spec \dz} \left( \hat{\ff}_D \hat{\otimes}_{\Q} \dr   \right) \right)
\;
\lrto
\;
\da_X^{\rm ar}(\ff)  \, \mbox{,}
$$
where the first and the second summands are mapped to the corresponding summands of~\eqref{ar_adel}. We denote
\begin{equation}  \label{def-E}
E = \left(
\left( \mathop{{\prod}'}\limits_{x \in D, D \nsubset C} \oo_{K_{x,D}}(\ff) \right)
\oplus
\left(
\prod\limits_{D \nsubset C, \, f(D)= \Spec \dz} \left( \hat{\ff}_D \hat{\otimes}_{\Q} \dr   \right) \right)
\right)
\;
\cap
\;
\left(
\da_{X, 01}^{\rm ar}(\ff)  + \da_{X, 02}^{\rm ar}(\ff)
\right)  \, \mbox{.}
\end{equation}

We recall the calculations for the quotient group of the group of Parshin-Beilinson adeles on a two-dimensional scheme.
Namely, from~\cite[Theorem~1]{Osip1} we have
\begin{multline}
G =
\left(
\mathop{{\prod}'}\limits_{x \in D, D \nsubset C} \oo_{K_{x,D}}(\ff) \right)
\;
\cap
\;
\left(
\da_{X, 01}(\ff)  + \da_{X, 02}(\ff)
\right) = \\  \label{form-G}
=
\prod\limits_{D \subset X, D \nsubset C}  \hat{\ff}_D
+ \prod\limits_{x \in U} \hat{\ff}_x + (\tau -\gamma)(A_C(\ff))
\, \mbox{,}
\end{multline}
where the intersection is taken inside the group $\da_X(\ff)$, the maps $\gamma$ and $\tau$ are the following natural maps:
\begin{gather}
\gamma \; : \; A_C(\ff)  \hookrightarrow \prod_{1 \le i \le w} K_{C_i}(\ff)  \hookrightarrow \da_X(\ff)  \nonumber \\
\label{second-line}
\tau \; : \;
A_C(\ff)
\hookrightarrow
\mathop{{\prod}'}_{x \in C} B_{x,C}(\ff)
\hookrightarrow \mathop{{\prod}'}_{x \in X}  K_x(\ff)
\hookrightarrow \da_X(\ff)
\end{gather}
and the first arrow in formula~\eqref{second-line}
is the composition of maps
$$A_C(\ff)
\simeq \Ker \Xi \lrto
\mathop{{\prod}'}_{x \in C} B_{x,C}(\ff)$$ which also given as the map
$z \mapsto v$ (see definition of the map $\Xi$ in formula~\eqref{Xi}).

A natural projection map $\da_X^{\rm ar}(\ff) \lrto \da_X(\ff)  $  induces a well-defined map:
$$
\theta \; : \;  E  \lrto G  \, \mbox{.}
$$

\medskip

First, we will prove that the map $\theta$ is injective.
We will need the following simple lemma.
\begin{lemma} \label{Inters}
Let $\g$ be a locally free sheaf on an irreducible algebraic curve $S$ over a number field $F$. Then we have
$$
(\hat{\g}_s  \otimes_{\oo_s} F(S))  \cap \left(\g_{\eta} \otimes_{\Q} \dr \right) = \g_{\eta} \, \mbox{,}
$$
where $s$ is a (closed) point  on $S$, $\oo_s$ is the local ring of the point $s, $ $\eta$ is the generic point, and the intersection is taken inside the group
$(\hat{\g}_s \otimes_{\oo_s} F(S) )  \hat{\otimes}_{\Q} \dr$.
\end{lemma}
\begin{proof} Clearly, we can change a sheaf $\g$ to a structure sheaf $\oo_S$.

Now this lemma is evident when a curve  $S= \mathbb{P}^1_{F}$, a point $s= [0 :1]$, since
$$
F((t)) \cap \left( F[t] \otimes_{\Q} \dr \right) = F[t]  \, \mbox{,}
$$
and we can change in the last formula $F[t]$ to $F(t)$.
For arbitrary curve $S$ and a point $s \in S$ we consider a morphism $r : S \lrto \mathbb{P}^1_{F}$ with the property  ${r(s)= [0:1]}$ and ${r(s') \ne [0:1]}$ for any other point $s'$ on $S$. (For example, $r = g^{-1}$, where $g \in H^0(S \setminus s, \oo_{S \setminus s} )$, $g$ is not constant: we use that $S \setminus s $ is an affine curve.) We have an embedding of fields $r^*:  F(t) \hookrightarrow F(S)$. Now the lemma follows after the tensor multiplication of a sequence
$$
0 \lrto F(t) \lrto F((t))  \oplus \left(F(t) \otimes_{\Q} \dr     \right)  \lrto F((t)) \otimes_{\Q}  \dr
$$
to the field $F(S)$ over the field $F(t)$.
 \end{proof}

\medskip

We  suppose now that that there is an element $(0 \oplus y) \in \Ker \theta $,
where
\begin{equation}  \label{cond-y}
y \in \prod\limits_{D \nsubset C, \, f(D)= \Spec \dz} \left( \hat{\ff}_D \hat{\otimes}_{\Q} \dr   \right)
\end{equation}
We will prove that $y = 0$. Hence we will obtain that the map $\theta$ is injective.

We have the following decomposition
\begin{equation}  \label{sum_y}
(0 \oplus y)= y_{01} + y_{02} + y_{\infty} \, \mbox{,}
\end{equation}
where $y_{01}= \prod\limits_D y_{01,D} \in \mathop{{\prod}'}\limits_{D \subset X} K_D(\ff)$,
$
y_{02} = \prod\limits_x y_{02,x} \in \mathop{{\prod}'}\limits_{x \in X} K_x(\ff)
$,
and
$y_{\infty} \in \ff_{\eta} \otimes_{\Q} \dr$. We consider two cases.

\medskip

{\em Case $1$.}  If $y_{\infty} =0$, then we have that $y_{01, C_i} =0$ for any $1 \le i \le w$, since the image of an element $y_{01, C_i} =0$ in
$K_{C_i}(\ff) \hat{\otimes}_{\Q} \dr$ will be equal zero because of condition~\eqref{cond-y} on~$y$.
Hence we obtain that $y_{02,x} =0 $ for any closed point $x \in C_i$, since the projection of the element $(0 \oplus y)$ to the group $K_{x,C_i}(\ff)$ equals zero. By considering a fibre of the map $f$ which contains a previous  point $x$ we obtain that  $y_{01, D}=0$ when $f (D) \ne \Spec  \dz$, since  the projection of the element $(0 \oplus y)$ to the group $K_{x,D}(\ff)$ equals zero. Now for any closed point $x'$ on a previous $D$ we obtain that $y_{02, x'}=0$. Now, by considering an integral one-dimensional closed subscheme $D'$ with $f(D')= \dz$ and $x' \in D'$ we have $y_{01,D'}=0$. Hence, by formula~\eqref{sum_y}, $y=0$. Case~$1$ is considered.

\medskip

{\em Case $2$.} If $y_{\infty} \ne 0$, then we have that the image of the element $y_{\infty}$ in $K_{C_i}(\ff) \hat{\otimes}_{\Q} \dr$ (for $1 \le i \le w$)
equals the image of an element $- y_{01,C_i}$ because of condition~\eqref{cond-y} on~$y$. Therefore, by Lemma~\ref{Inters} applied to the curve $X_{\Q}$, we have
that $y_{\infty} = - y_{01,C_i} \in \ff_{\eta}$. Hence, by similar arguments as in the previous case, we obtain that $y_{02,x}=- y_{01,C_i} $ for any closed point $x \in C_i$. Therefore $y_{01, D}= - y_{02,x}$ for any one-dimensional integral closed subscheme $D$ with $f(D) \ne \Spec \dz$ and
$D$ contains a previous point $x$. Now for any closed point $x'$ on a previous $D$ we obtain that $y_{02, x'}= - y_{01,D}$. Finally,
for any integral one-dimensional closed subscheme $D'$ with $f(D')= \dz$ and $x' \in D'$ we have $y_{01, D'}= - y_{02, x'}$. Thus,
$y_{01,D'}= - y_{\infty} \in \ff_{\eta}$. Therefore, by formula~\eqref{sum_y}, we obtain $y =0$.
Case~$2$ is considered.

\medskip

Thus we have proved that the map $\theta$ is injective.

\medskip
We want to prove that
\begin{equation}   \label{E-form}
E = \prod\limits_{D \subset X, D \nsubset C}  \hat{\ff}_D
+ \prod\limits_{x \in U} \hat{\ff}_x \, \mbox{.}
\end{equation}
It will gives the first non-zero term of exact sequence~\eqref{main-seq} from Theorem~\ref{main-Th}.

We have the following commutative diagram:
\begin{equation}   \label{diagr}
\xymatrix{
&
 \prod\limits_{D \subset X, D \nsubset C}  \hat{\ff}_D
+ \prod\limits_{x \in U} \hat{\ff}_x  \,
\ar[dl]_-{\alpha}
\ar@{^{(}->}[d]
\ar@{^{(}->}[r]
&
\da_X^{\rm ar}(\ff)  \\
G &
E \, \ar@{_{(}->}[l]_{\theta}
\ar@{^{(}->}[ur]
 & \\
}
\end{equation}
where (we recall): the summand   $\prod\limits_{D \subset X, D \nsubset C}  \hat{\ff}_D$ is mapped in the both summands of the group $\da_X^{\rm ar} (\ff)$, see formula~\eqref{ar_adel}.

From this diagram and formula~\eqref{form-G} it follows that we prove  equality~\eqref{E-form} if we prove
 that
the group $\Image \alpha$ contains a subgroup
\begin{equation}  \label{group-H}
H = (\tau - \gamma)(A_C(\ff)) \cap \Image \theta \mbox{.}
\end{equation}
We will prove it now. We suppose that
\begin{equation}  \label{H-form}
H \, \ni \,  a = \theta(b)  \, \mbox{.}
\end{equation}
We have that $b \in \da_{X, 01}^{\rm ar} (\ff) + \da_{X, 02}^{\rm ar}(\ff)$.
 Therefore
 $$
 b = b_{01}  + b_{02} + b_{\infty}   \, \mbox{,}
 $$
 where
 $b_{01}=\prod\limits_D b_{01,D} \in \mathop{{\prod}'}\limits_{D \subset X} K_D(\ff)$,
$
b_{02} = \prod\limits_x b_{02,x} \in \mathop{{\prod}'}\limits_{x \in X} K_x(\ff)
$,
and
$b_{\infty} \in \ff_{\eta} \otimes_{\Q} \dr$.
 We consider again two cases.

 \medskip

{\em Case $1$.} Suppose that there is an integer $1 \le i \le w $ such that an element  ${b_{01, C_i}=0}$. Therefore, since by formula~\eqref{def-E}   the image of the element $b$ in the group
$K_{C_i}(\ff)  \hat{\otimes}_{\Q}  \dr$ equals zero, we have that $b_{\infty}=0$. This implies (again by formula~\eqref{def-E})
that $b_{01, C_j}=0$ for any integer $1 \le j \le w$.
For any closed point $x \in C_j$ the image of the element $b$  in the group $K_{x,C_j}(\ff)$ equals zero (by formula~\eqref{def-E}). Therefore from $b_{01, C_j}=0$ we have $b_{02,x}=0$.
Thus we obtained that
$$
b = \left( \prod\limits_{D \nsubset C} b_{01,D}  \right) +  \left(  \prod\limits_{x \in U}  b_{02,x}                 \right)   \, \mbox.
$$
Suppose that a one-dimensional integral closed subscheme $D \subset X$ is not contained in a fibre of the map $f$. Then the image of
the element $b_{01, D}$ in the group $K_D(\ff) \hat{\otimes}_{\Q}  \dr $ belongs to the subgroup $ \hat{\ff}_{D}  \hat{\otimes}_{\Q}  \dr $,
because of formula~\eqref{def-E} and $b_{\infty} =0$. Hence we have that $b_{01,D} \in \hat{\ff}_{D}$.
Now suppose that a one-dimensional integral closed subscheme $D \subset X$ is contained in a fibre of the map $f$.
Then there is a closed point $x \in D \cap C$. From formula~\eqref{def-E}  we have that the image of the element $b$ in the group $K_{x,D}(\ff)$
belongs to the subgroup $\oo_{K_{x,D}}(\ff)$.
Since $b_{02,x}=0$, we obtain that $b_{01,D}  \in \hat{\ff}_D$.
Hence, for any closed point $x \in U$ we obtain $b_{02,x} \in \hat{\ff}_x$, since for any one-dimensional integral closed subscheme $D \ni x$ we have from $\theta(b) \in H$  that
$b_{02,x} + b_{01,D} =0$
inside the group $K_{x,D}(\ff)$ (i.e. the image of the element $b_{02,x}$ belongs to the subgroup $\oo_{K_{x,D}}(\ff)$). Thus, we proved that
$$
b \in \prod\limits_{D \subset X, D \nsubset C}  \hat{\ff}_D
+ \prod\limits_{x \in U} \hat{\ff}_x   \, \mbox{.}
$$
This implies that $a= \theta(b)= \alpha(b)$ (see formula~\eqref{H-form}).
Case $1$ is considered.

\medskip

{\em Case $2$.} Now suppose that for any integer $1 \le i \le w$ we have $b_{01, C_i} \ne 0$.
Since for any integer $1 \le i \le w$  the image of the element $b$ in the group $K_{C_i}(\ff) \hat{\otimes}_{\Q}  \dr$ equals zero (see formula~\eqref{def-E}),
we have that $b_{01,C_i} + b_{\infty} =0$
inside the group $K_{C_i}(\ff) \hat{\otimes}_{\Q}  \dr$. Therefore, by Lemma~\ref{Inters} applied to the curve $X_{\Q}$, we obtain that
$$
b_{01, C_i} = - b_{\infty} \, \in  \, \ff_{\eta}  \, \mbox{.}
$$
For any closed point $x \in C_i$ we have $b_{02,x}= -b_{01, C_i}$, since by formula~\eqref{def-E} the image of the element $b$ in the group
$K_{x,C_i}(\ff)$ equals zero. Thus we have
\begin{equation}  \label{on_U}
  \theta \left(  \prod_D b_{01, D}  + \prod_x b_{02,x} + b_{\infty}   \right)
\, \in \,  (\tau - \gamma)(A_C(\ff))
  \, \mbox{,}
\end{equation}
where
 for any $1 \le i \le w$ and for any $x \in C_i$  we have $b_{01, C_i} =- b_{02,x}= -b_{\infty} \in \ff_{\eta}$.
 From formula~\eqref{on_U} we obtain an equality in the group $\da_U(\ff)$, which is the Parshin-Beilinson adeic group of $\ff$ on $U$:
\begin{equation}   \label{1-coc}
\prod_{D \nsubset C} b_{01,D}  + \prod_{x \in U} b_{02,x} =0  \, \mbox{.}
\end{equation}
Since the scheme $U$
is affine, we have $H^1(U, \oo_U)=0$. Hence and from formula~\eqref{1-coc} we obtain that a $1$-cocycle
$\left( \prod\limits_{D \nsubset C} b_{01,D}, - \prod\limits_{x \in U} b_{02,x}, 0  \right)$
 in the adelic complex on $U$ is a $1$-coboundary. This means that there are elements $g \in \ff_{\eta}$, $e_{2,x} \in \hat{\ff}_x$ for any closed point $x \in U$, and $e_{1,D} \in \hat{\ff}_D$ for any integral $1$-dimensional closed subscheme $D$ on $U$ such that
 \begin{equation}  \label{Bud-Gor}
 b_{01, D}= e_{1,D} -g  \quad \mbox{,}  \qquad b_{02,x}= e_{2,x} +g  \quad \mbox{,}  \qquad e_{1,D} + e_{2,x} =0  \, \mbox{,}
 \end{equation}
 where the first equality is in the group $K_D(\ff)$, the second equality is in the group $K_x(\ff)$, and the third equality is in the group
 $\oo_{x,D}(\ff)$.
 Thus, we obtain the following equality:
 \begin{equation}  \label{b-e}
 b = \left(\prod_D b_{01,D} \right) + \left( \prod_{x} b_{02,x} \right)  + b_{\infty} = \left( \prod_{D} e_{01,D} \right)  +
 \left(\prod_{x} e_{02,x} \right)  + e_{\infty}  \, \mbox{,}
 \end{equation}
 where $e_{01,D} = e_{1, D} = b_{01, D} +g  \in \hat{\ff}_D$  for $D  \not{\hspace{-0.1cm} \subset} C$, $e_{02,x}= e_{2,x} = b_{02,x} -g \,\in \hat{\ff}_x$  for $x \in U$, and
 $$e_{\infty}= - e_{01,C_i}= e_{02,x}= b_{\infty} -g = -b_{01,C_i} -g
 = b_{02,x} - g
   \; \, \in  \; \ff_{\eta}$$ for $x \in C_i$ and $1 \le i \le w$.
 \begin{nt}{ \em
 We did not use the last equality in formula~\eqref{Bud-Gor}. This equality means that
 $$
 \prod_{D \subset U} e_{01,D} = - \prod_{x \in U} e_{02,x}   \; \, \in \; \, \da_{1, U}(\ff)  \cap \da_{2, U}(\ff)  \, \mbox{,}
 $$
 where the subgroups $\da_{1, U}(\ff)$ and  $\da_{2, U}(\ff)$ were defined by formula~\eqref{ad-subgr}, and the intersection is taken in the group
 $\da_{U}(\ff)$. This intersection contains the subgroup $H^0(U, \ff \mid_U)$, but, as it was shown in~\cite[\S~6]{BG},  it is not equal to $H^0(U, \ff \mid_U)$. (More exactly, in~\cite[\S~6]{BG} a corresponding  example  was constructed when $U$ is an affine regular surface over a countable field, but the same reasonings work also in our case.) }
 \end{nt}
Now from formulas~\eqref{def-E} and~\eqref{b-e} we obtain that $e_{\infty} \in A_{U_{\Q}}(\ff)= H^0(U_{\Q}, \ff_{\Q}  \mid_{U_{\Q}} )$.
For any $D  \not{\hspace{-0.1cm} \subset} C$ we have $e_{01,D}  \in \hat{\ff}_D$. Therefore we obtain
from formula~\eqref{def-E} that
for any $x \in C_i$ with $1 \le i \le w$
 the image of an element $e_{02,x}$ in the group $K_{x,D}(\ff)$ belongs to the subgroup
$\oo_{K_{x,D}}(\ff)$.  Using $e_{\infty} = e_{02,x}$ we obtain
$$
e_{\infty}= - e_{01,C_i}= e_{02,x}  \; \in \; H^0(U, \ff \mid_U)  \, \mbox{.}
$$
Thus we have
$$
b =  \left( \prod_{D} e_{01,D} \right)  +
 \left(\prod_{x} e_{02,x} \right)  + e_{\infty} =
 \left( \prod_{D \nsubset C} f_{01,D} \right) + \left(  \prod_{x \in U}  f_{02,x}                   \right)   \, \mbox{,}
$$
where $f_{02,x} = e_{02,x} - e_{\infty}  \, \in \hat{\ff}_x$ for $x \in U$,
and $f_{01, D} = e_{01,D} + e_{\infty}  \, \in \hat{\ff}_D$ for $D  \not{\hspace{-0.1cm} \subset} C$.
This implies that $a= \theta(b)= \alpha(b)$ (see formula~\eqref{H-form}).
Case $2$ is considered.

\medskip

Thus we have proved equality~\eqref{E-form}.

This gives the first non-zero term of exact sequence~\eqref{main-seq} from Theorem~\ref{main-Th}.

\subsection{Last non-zero term in exact sequence~\eqref{main-seq}}  \label{last_non-zero_term}
In this section we obtain the last non-zero term in exact sequence~\eqref{main-seq}.
To do it,  we will make the following calculations:
\begin{multline*}
\frac{\Phi}{
\left( \mathop{{\prod}'}\limits_{x \in D, D \nsubset C} \oo_{K_{x,D}}(\ff) \right)
\oplus
\left(
\prod\limits_{D \nsubset C, \, f(D)= \Spec \dz} \left( \hat{\ff}_D \hat{\otimes}_{\Q} \dr   \right)
\right)} \simeq  \\ \simeq
\frac{\da_X^{\rm ar}(\ff)}{\da_{X, 01}^{\rm ar}(\ff)  + \da_{X, 02}^{\rm ar}(\ff)
+
\left(
\left( \mathop{{\prod}'}\limits_{x \in D, D \nsubset C} \oo_{K_{x,D}}(\ff) \right)
\oplus
\left(
\prod\limits_{D \nsubset C, \, f(D)= \Spec \dz} \left( \hat{\ff}_D \hat{\otimes}_{\Q} \dr   \right)
\right)
\right) } \simeq
\\ \simeq
\frac{\Phi_1 \oplus \Phi_2}{\beta(\Phi_3)}  \, \mbox{,}
\end{multline*}
where
\begin{gather*} \Phi_1 =
\frac{\da_X(\ff)}{\da_{X, 02}(\ff) + \mathop{{\prod}'}\limits_{D, \, f(D) \ne \Spec \dz} K_D(\ff) + \mathop{{\prod}'}\limits_{x \in D, D \nsubset C} \oo_{K_{x,D}}(\ff)  }  \, \mbox{,}
\\
\Phi_2 =
\frac{\da_{X_{\Q}}(\ff) \otimes_{\Q} \dr}
{  \left( \ff_{\eta} \otimes_{\Q} \dr \right) +
\prod\limits_{D \nsubset C, \, f(D)= \Spec \dz} \left( \hat{\ff}_D \hat{\otimes}_{\Q} \dr    \right)}
\;
\mbox{,} \qquad
\Phi_3 = \mathop{{\prod}'}\limits_{D, \, f(D) = \Spec \dz} K_D(\ff) \, \mbox{,}
\end{gather*}
and $\beta$ is a natural map from the group $\Phi_3$ to the group $\Phi_1 \oplus \Phi_2$ (to the both parts of this direct sum).

\medskip

Now we prove that
\begin{equation}  \label{eq-zero}
\frac{\da_X(\ff)}{\da_{X, 02}(\ff) + \mathop{{\prod}'}\limits_{D, \, f(D) \ne \Spec \dz} K_D(\ff) + \mathop{{\prod}'}\limits_{x \in D, D \nsubset C} \oo_{K_{x,D}}(\ff) + \Phi_3 } \; \simeq  \; 0 \, \mbox{.}
\end{equation}
Indeed, from~\cite[Theorem~1]{Osip1} we immediately obtain that the left hand side of expression from formula~\eqref{eq-zero} equals to
\begin{equation}  \label{K_x_C}
\frac{ \prod\limits_{ 1 \le i \le w }  \mathop{{\prod}'}\limits_{x \in C_i}   K_{x, C_i}(\ff))}{
 \prod\limits_{1 \le i \le w} K_{C_i}(\ff)   + \mathop{{\prod}'}\limits_{x \in C} B_{x,C}(\ff)   }   \, \mbox{,}
\end{equation}
which equals zero, because it is isomorphic to  the group $\varinjlim\limits _{n} \varprojlim\limits_{m < n}
 H^1 \left( X, \ff_{nm}\right)    $ via inductive and projective limits of  adelic complexes
 of coherent sheaves \linebreak ${\ff_{nm} = \ff \otimes_{\oo_X} \left( \oo_X(nC)/   \oo_X(mC) \right)}$ on the scheme $Y_{n-m}= (C, \oo_X / J_C^{n-m})$
 with to\-po\-lo\-gi\-cal space $C$ and the structure sheaf $\oo_X / J_C^{n-m}$, where $J_C$ is the ideal sheaf of the subscheme $C$ on $X$, and
 $H^1 \left( X, \ff_{nm}\right) = H^1 \left( Y_{n-m}, \ff_{nm}\right) =0$, because $Y_{n-m}$ is an affine scheme (see similar reasonings also
 in~\cite[Lemma~7]{OsipPar2})\footnote{We recall that $\oo_X(nC)$ is a reflexive
 torsion free  coherent subsheaf of the constant sheaf of the field of rational functions on $X$, and this subsheaf consists of elements of $j_* \oo_U$ which have discrete valuations given by subschemes $C_i$ (where $1 \le i \le w$) greater or equal to $-n$, see also the beginning of Section~3.2.1 from~\cite{Osip1}.}.
 \begin{nt} {\em
 We note that reduction of formula~\eqref{eq-zero} to equality to zero of formula~\eqref{K_x_C} follows easily also from $H^2(X, j_* j^* \ff) =0$ and the adelic complex for the sheaf $j_* j^* \ff$ on~$X$. }
 \end{nt}

\medskip

Now we claim that
\begin{equation}  \label{psi}
\frac{\Phi_1 \oplus \Phi_2}{\beta(\Phi_3)} \simeq
\frac{\Phi_2
}{
\lambda \left(\Phi_3 \cap \left( \da_{X, 02}(\ff) + \mathop{{\prod}'}\limits_{D, \, f(D) \ne \Spec \dz} K_D(\ff) + \mathop{{\prod}'}\limits_{x \in D, D \nsubset C} \oo_{K_{x,D}}(\ff)       \right)   \right)
}  \, \mbox{,}
\end{equation}
where the intersection is taken inside the group $\da_X(\ff)$, and the map $\lambda$ is a natural map from the group $\Phi_3$
to the group $\Phi_2$.

 We construct the map $\psi$ from the group in the left hand side to the group in the  right hand side of
formula~\eqref{psi} in the following way. Let an element $w$ be from the group $(\Phi_1 \oplus \Phi_2)/\beta(\Phi_3)$. We will define an element $\psi(w)$.
Let $x \oplus y$ be a lift of the element $w$, where $x \in \Phi_1$ and $y \in \Phi_2$. Let $\tilde{x} \in \da_{X}(\ff)$ be a lift of the element $x$. By formula~\eqref{eq-zero} there is an element $g \in \Phi_3$ such that
\begin{equation}   \label{cond_g}
g + \tilde{x}  \;  \in   \;  \da_{X, 02}(\ff) + \mathop{{\prod}'}\limits_{D, \, f(D) \ne \Spec \dz} K_D(\ff) + \mathop{{\prod}'}\limits_{x \in D, D \nsubset C} \oo_{K_{x,D}}(\ff)  \, \mbox{.}
\end{equation}
Now, by definition,  $\psi(w) =  y + \lambda(g) $.
Clearly, the map $\psi$ is well-defined, since it is clear that $\psi$ does not depend on the choice of the lifts of elements $w$ and $x$,
and if $g_1 $ and $g_2$ are two elements form the group $\Phi_3$ with condition as in formula~\eqref{cond_g}, then
$$
g_1 - g_2  \; \in \; \Phi_3 \cap \left( \da_{X, 02}(\ff) + \mathop{{\prod}'}\limits_{D, \, f(D) \ne \Spec \dz} K_D(\ff) + \mathop{{\prod}'}\limits_{x \in D, D \nsubset C} \oo_{K_{x,D}}(\ff)       \right)  \, \mbox{.}
$$
It is also clear that $\psi$
is a surjective map. Moreover, it is easy to see that $\psi$ is an injective map. Thus $\psi$
is an isomorphism.

\medskip

Now we will prove that
\begin{equation}  \label{lambda}
\lambda \left(\Phi_3 \cap \left( \da_{X, 02}(\ff) + \mathop{{\prod}'}\limits_{D, \, f(D) \ne \Spec \dz} K_D(\ff) + \mathop{{\prod}'}\limits_{x \in D, D \nsubset C} \oo_{K_{x,D}}(\ff)    \right)   \right)  = \lambda (A_C(\ff))  \, \mbox{.}
\end{equation}
Indeed, we suppose that
\begin{equation}  \label{eq-c}
c_{01}' = c_{02} + c_{01}'' + c_{12}   \; \in \;  \Phi_3 \cap \left( \da_{X, 02}(\ff) + \mathop{{\prod}'}\limits_{D, \, f(D) \ne \Spec \dz} K_D(\ff) + \mathop{{\prod}'}\limits_{x \in D, D \nsubset C} \oo_{K_{x,D}}(\ff)    \right) \, \mbox{,}
\end{equation}
where $c_{01}' = \prod\limits_D c_{01,D}  \in \Phi_3$, $c_{01}'' = \prod\limits_D c_{01,D}  \in \mathop{{\prod}'}\limits_{D, \, f(D)  \ne \Spec \dz} K_D(\ff)$, ${c_{02}= \prod\limits_x c_{02,x}  \in \da_{X, 02}(\ff)}$, ${c_{12}= \prod\limits_{x \in D} c_{12,x,D}  \in
\mathop{{\prod}'}\limits_{x \in D, \, D \nsubset C} \oo_{K_{x,D}}(\ff)}$.

An equality
 \begin{equation}   \label{eq-c-expl}
 {c_{02} + c_{01}'' -c_{01}' + c_{12} =0}
 \end{equation}
 restricted to the open affine subset $U$ defines a $1$-cocycle
 $(c_{01}'' -c_{01}', - c_{02}, c_{12})$
 in the adelic complex of the sheaf $\ff$ on~$U$. Since $H^1(U, \oo_U)=0$, this $1$-cocycle is a $1$-coboundary.
This means that there exist  elements $d_{1,D} \in \hat{\ff}_D$ for any $1$-dimensional integral closed subscheme  $D \subset X$  such that $D \not \hspace{-0.1cm} \subset C$,
elements $d_{2,x} \in \hat{\ff}_x$ for any $x \in U$, and an element $h \in \ff_{\eta}$ such that
\begin{gather*}
d_{1,D} +h = c_{01,D} \qquad \mbox{for any $D$ with $f(D) \ne \Spec \dz$,} \\
d_{1,D} -h = c_{01,D} \qquad \mbox{for any $D$ with $D \not \hspace{-0.1cm} \subset C$ and $f(D) = \Spec \dz$,} \\
d_{2,x} - h= c_{02,x} \qquad  \mbox{for any $x \in U$.}
\end{gather*}
 We define also  elements  $d_{01,D}=d_{1,D}$ and $d_{02,x}=d_{2,x}$ for $D$ and $x$ as above, and
 \begin{gather*}
 d_{01,C_i}= c_{01, C_i} + h \qquad \mbox{for any $1 \le i \le w$,} \\
d_{02,x} = c_{02,x} +h \qquad \mbox{for any $x \in C$,} \\
d_{12,x,D}= c_{12,x,D}  \in \oo_{K_{x,D}}(\ff) \qquad \mbox{for any $x \in D$ with $D \not \hspace{-0,1cm} \subset C$.}
\end{gather*}
Then we have for elements ``$d$'' the same equation as for elements ``$c$'' in formula~\eqref{eq-c-expl}. Hence and from above conditions for elements ``$d$'' we obtain
 that $\prod\limits_{1 \le i \le w} d_{01, C_i}  \in A_C(\ff)$.  Therefore the group which is intersection in formula~\eqref{eq-c} is contained in the group
 $$A_C(\ff) + \prod\limits_{D \nsubset C, \, f(D) = \Spec \dz} \hat{\ff}_D  + \quad (\mbox{image of} \quad \ff_{\eta} \quad \mbox{in} \quad \Phi_3)  \mbox{.}$$
Besides, the group $A_C(\ff)$ is contained in the group which is intersection in formula~\eqref{eq-c}.
Indeed, it is evident that $A_C(\ff)  \subset \Phi_3$, and we have also
$$A_C(\ff)  \subset \da_{X, 02}(\ff)  + \mathop{{\prod}'}\limits_{x \in D, D \nsubset C} \oo_{K_{x,D}}(\ff)   $$
by means of the map $\tau$ (see formulas~\eqref{second-line} and~\eqref{B-form}).
Hence, bearing in mind definition of the group $\Phi_2$, we obtain equality~\eqref{lambda}.

Thus, from formulas~\eqref{psi} and~\eqref{lambda} we have that
$$
\frac{\Phi_1 \oplus \Phi_2}{\beta(\Phi_3)} \simeq
\frac{\Phi_2
}{\lambda(A_C(\ff))}  \, \mbox{.}
$$

From calculation of adelic quotient group on the algebraic curve $X_{\Q}$, namely from~\cite[Remark~1]{Osip1} applied to the algebraic curve $X_{\Q}$
and the open affine subset $U_{\Q}$, we obtain
\begin{equation}  \label{phi_2}
\Phi_2   \;
\simeq
\;
\frac{ \prod\limits_{1 \le i \le w} \left( K_{C_i}(\ff) \hat{\otimes}_{\Q} \dr \right) }{
A_{U_{\Q}}(\ff) \otimes_{\Q}  \dr
} \, \mbox{.}
\end{equation}
Hence we obtain that
$$
\frac{\Phi_1 \oplus \Phi_2}{\beta(\Phi_3)} \simeq \frac{ \prod\limits_{1 \le i \le w} \left( K_{C_i}(\ff) \hat{\otimes}_{\Q} \dr \right) }{
A_C(\ff) + \left( A_{U_{\Q}}(\ff) \otimes_{\Q}  \dr  \right)
} \, \mbox{.}
$$

This gives the last non-zero term of exact sequence~\eqref{main-seq} from Theorem~\ref{main-Th}.
Thus we have proved this theorem.

\begin{nt} \label{C-curve}{\em
From the proof that the group in formula~\eqref{K_x_C} equals zero and from
construction of the isomorphism $\psi$ in formula~\eqref{psi}
it follows that the image of the group $\prod\limits_{1 \le i \le w}
\mathop{{\prod}'}\limits_{x \in C_i} \oo_{K_{x,C_i}}(\ff)$ in the group $\Phi$ is mapped  to the image of the group $\prod\limits_{1 \le i \le w} \hat{\ff}_{C_i}$ in the group
$
\frac{ \prod\limits_{1 \le i \le w} \left( K_{C_i}(\ff) \hat{\otimes}_{\Q} \dr \right) }{
A_C(\ff) + \left( A_{U_{\Q}}(\ff) \otimes_{\Q}  \dr  \right)
} $
in exact sequence~\eqref{main-seq}.
}
\end{nt}

\section{Connection with the first cohomology groups}  \label{appl}
In this section we will say more on a group which is the last non-zero term in exact sequence~\eqref{main-seq} from Theorem~\ref{main-Th}.
In particular, we relate it with the first cohomology groups of sheaves $\ff(nC)$ on $X$, where $\ff(nc)= \ff \otimes_{\oo_X} \oo_X(nC)$.
We demand also additionally that $C$ is an ample Cariter divisor on $X$.

For any integer $m \ge 0$ we consider a quotient group of
$
\frac{ \prod\limits_{1 \le i \le w} \left( K_{C_i}(\ff) \hat{\otimes}_{\Q} \dr \right) }{
A_C(\ff) + \left( A_{U_{\Q}}(\ff) \otimes_{\Q}  \dr  \right)}:
$
$$
\Theta_m(\ff) = \frac{ \prod\limits_{1 \le i \le w} \left( K_{C_i}(\ff) \hat{\otimes}_{\Q} \dr \right) }{
A_C(\ff) + \left( A_{U_{\Q}}(\ff) \otimes_{\Q}  \dr \right) + \left( \prod\limits_{1 \le i \le w} \widehat{\ff(-mC)}_{C_i}
\hat{\otimes}_{\Q}  \dr
 \right)}   \mbox{.}
$$

\begin{prop}  \label{theta}
We have the following properties.
\begin{enumerate}
\item
For any integer $m \ge 0$
the group
$$\Theta_m(\ff)
\simeq H^1(X, \ff(-mC))  \otimes_{\dz} \mathbb{T} \simeq \mathbb{T}^{{\rm \, rank}(H^1(X, \, \ff(-mC) ))}
$$
 is a finite direct product of copies of  $\mathbb{T}$ (we recall that $\mathbb{T} \simeq \dr/\dz$).
\item There is $m_0 \ge 0$ (which depends on $\ff$) such that for any integer $m \ge m_0$ we have an exact sequence
$$
 0 \lrto  \mathbb{T}^{\, r l} \lrto  \Theta_{m+1}(\ff) \lrto  \Theta_m(\ff)  \lrto 0  \, \mbox{,}
$$
where $r$ is the rank of the locally free sheaf $\ff$, and $l$ is the degree of the finite morphism $f \mid_C \, : \, C  \to \Spec \dz$.
\end{enumerate}
\end{prop}
\begin{proof}
{\em 1.}
Since $C$ is an affine $1$-dimensional scheme and $f \mid_C \, : \, C  \to \Spec \dz$ is a flat morphism, we have an embedding of a free Abelian group
into a $\Q$-vector space:
\begin{multline*}
A_C(\ff)/ ( \varprojlim\limits_{j < -m} H^0(X, \, \ff(-mC)/ \ff(jC)) ) \simeq \varinjlim\limits_{n >  -m} H^0(X, \ff(nC)/ \ff(-mC))  \simeq \\ \simeq
\bigoplus\limits_{n > -m} H^0(X, \ff(nC)/ \ff((n-1)C))  \hookrightarrow   \bigoplus\limits_{n > -m} H^0(X, \ff(nC)/ \ff((n-1)C)) \otimes_{\dz} \Q
\simeq \\
\simeq
\left( \prod\limits_{1 \le i \le w} K_{C_i}(\ff)  \right) /
\left(
\prod\limits_{1 \le i \le w} \widehat{\ff(-mC)}_{C_i}
\right)
\, \mbox{.}
\end{multline*}
(We note that the decomposition into the direct sum in this formula is not canonical, but we fix it.)

Besides, we have an embedding of a free Abelian group into a $\Q$-vector space:
$$
H^0 (U, \ff \mid_U) \simeq
\varinjlim\limits_n H^0(X, \ff(nC))  \hookrightarrow
\varinjlim\limits_n H^0(X, \ff(nC))  \otimes_{\dz} \Q \simeq
H^0(U_{\Q}, \ff_{\Q}  \mid_{U_{\Q}}) \simeq A_{U_{\Q}}(\ff) \mbox{.}
$$
Hence we obtain that
\begin{equation}  \label{integer-form}
\Theta_m(\ff) \simeq \frac{\frac{A_C(\ff)}{H^0(U, \, \ff \mid_U) + \varprojlim\limits_{j < -m} H^0(X, \, \ff(-mC)/ \ff(j)) } \otimes_{\dz} \dr}{A_C(\ff)}
\end{equation}

By~\cite[Theorem~9]{Osi}   for any integer $m$ we have
$$
\frac{A_C(\ff)}{H^0(U, \, \ff \mid_U) + \varprojlim\limits_{j < -m} H^0(X, \, \ff(-mC)/ \ff(jC)) } \simeq H^1(X, \ff(-mC)) \, \mbox{,}
$$
and we know that these groups are finitely generated Abelian groups (the last also follows from reasonings in the next paragraph).

Since $C$ is an ample divisor, there is $n_0 > 0$ such that for any $n > n_0$
we have a surjective map
\begin{equation}  \label{sur}
H^0(X, \ff(nC))  \lrto  H^0(X, \ff(nC)/ \ff((n-1)C))  \, \mbox{.}
\end{equation}
This means that the  Abelian group $H^0 (U, \ff \mid_U)$  is mapped surjectively onto the
{Abelian group} $\bigoplus\limits_{n > n_0} H^0(X, \ff(nC)/ \ff((n-1)C))$ by means of maps given by formula~\eqref{sur},
and the kernel is a finitely generated Abelian group.

This all together gives the proof of the first statement of the proposition.

{\em 2.} Exact sequence of locally free sheaves
$$
0 \lrto \ff((-m-1)C)  \lrto \ff(-mC)  \lrto \ff(-mC) / \ff((-m-1)C)  \lrto 0
$$
induces  the long exact sequence of cohomology groups
\begin{multline*}
H^0(X, \ff(-mC))  \lrto H^0(X, \ff(-mC)/ \ff((-m-1)C)) \lrto \\
\lrto
 H^1(X, \ff((-m-1)C)  \lrto H^1(X, \ff(-mC))  \lrto 0  \, \mbox{.}
\end{multline*}
Besides, there is an integer $m_0 \ge 0$ such that for any $m \ge m_0$ the group $H^0(X, \ff(-mC))$
is zero, because  the free Abelian group $H^0(X, \ff)$ is finitely generated and
$$
\bigcap_{j \ge 0} H^0(X, \ff(-jC))  = 0  \, \mbox{.}
$$

For any integer $n$ the finitely generated free Abelian group ${H^0(X, \ff(nC)/ \ff((n-1)C))}$ is a cocompact lattice in an $\dr$-vector space
$$H^0(X, \ff(nC)/ \ff((n-1)C)) \otimes_{\dz} \dr \simeq
\left(\prod\limits_{1 \le i \le w} \widehat{\ff(nC)}_{C_i}
\hat{\otimes}_{\Q}  \dr \right) / \left( \prod\limits_{1 \le i \le w} \widehat{\ff((n-1)C)}_{C_i}
\hat{\otimes}_{\Q}  \dr \right) $$
such that the quotient group of this vector space by the lattice is isomorphic to $\mathbb{T}^{\, rl}$.

This gives the proof of the second statement of the proposition.
\end{proof}

From the proof of Proposition~\ref{theta} (see formula~\eqref{integer-form}) we obtain that
\begin{equation}  \label{comp}
\frac{ \prod\limits_{1 \le i \le w} \left( K_{C_i}(\ff) \hat{\otimes}_{\Q} \dr \right) }{
A_C(\ff) + \left( A_{U_{\Q}}(\ff) \otimes_{\Q}  \dr  \right)}  \;
\simeq  \;  \varprojlim\limits_{m \ge 0} \Theta_m(\ff)  \, \mbox{,}
\end{equation}
This  evidently implies that the group from formula~\eqref{comp} is compact.

\medskip

\begin{nt}{\em  \label{Weng}
It is naturally to consider the subgroup $\da_{X, 12}^{\rm ar}(\ff)$ of the arithmetic adelic group~$\da_{X}^{\rm ar}(\ff)$:
$$
\da_{X,12}^{\rm ar} (\ff) \eqdef \da_{X,12}(\ff) \; \oplus \mathop{{\prod}}_{D, f(D) =\Spec \dz}  \hat{\ff}_D \hat{\otimes}_{\Q} \dr \; = \;
\mathop{{\prod}'}_{x \in D} \oo_{K_{x,D}}(\ff) \;
\oplus \mathop{{\prod}}_{D, f(D) \; = \; \Spec \dz}  \hat{\ff}_D \hat{\otimes}_{\Q} \dr   \, \mbox{.}
$$
K.~Sugahara and L.~Weng introduced in~\cite{SW1} the following groups:
\begin{gather*}
H^0_{\rm ar}(X, \ff) = \da_{X,01}^{\rm ar}(\ff)  \cap \da_{X, 02}^{\rm ar}(\ff)  \cap \da_{X, 12}^{\rm ar} (\ff)  \qquad \mbox{and} \\
H^2_{\rm ar}(X, \ff) = \da_{X}^{\rm ar}(\ff) / \left(\da_{X,01}^{\rm ar}(\ff)  + \da_{X, 02}^{\rm ar}(\ff)  + \da_{X, 12}^{\rm ar} (\ff) \right) \mbox{.}
\end{gather*}
From Theorem~\ref{main-Th}, Remark~\ref{C-curve} and Proposition~\ref{theta} we immediately obtain that
$$
H^2_{\rm ar}(X, \ff) \simeq \Theta_0(\ff) \simeq  H^1(X, \ff)  \otimes_{\dz} \mathbb{T} \, \mbox{.}
$$
Besides, from Lemma~\ref{Inters}  it is easy to see that
$$
H^0_{\rm ar}(X, \ff) \simeq H^0(X, \ff)  \, \mbox{.}
$$

We note also that  a topological duality between the topological groups $H^i_{\rm ar}$ and $H^{2-i}_{\rm ar}$ was defined in~\cite{SW1,SW2} by means the properties of topology of inductive and projective limits on adelic groups (we don't give here the definition of the group $H^1_{\rm ar}$ and don't specify the sheaves in cohomology groups which appear in the duality). More on this topological duality it will be written also in a subsequent paper~\cite{Osip2}.
}
\end{nt}

\begin{nt}{\em \label{Arakelov}
By formula~\eqref{comp} (and Proposition~\ref{theta}) we obtained that the last non-zero term in exact sequence~\eqref{main-seq} from Theorem~\ref{main-Th} equals to the projective limit of 
groups which are finite direct products of copies of
 $\dt$. This gives an interesting analogy
with values of theta-functions (or, in other words, $\theta$-invariants) constructed by a cocompact lattice $E$ in a finite-dimensional $\dr$-vector space $\prod\limits_{\upsilon} K_{\upsilon}$
 which appear in exact sequence~\eqref{one-dim} in the Introduction. By these $\theta$-invariants one obtains   the Riemann-Roch theorem in  the form of Poisson summation formula in one-dimensional Arakelov geometry, and this was invented by many people, see~\cite{Bo}.

Indeed, let $\overline{M}$ be a Hermitian vector bundle over the ``arithmetic curve'' $\Spec E$. Then, by definition,
$$
h^0_{\theta}(\overline{M})= \log (\sum_{v \in M} \exp(- \pi \parallel v \parallel^2_{g_* \overline{M}}) \, \mbox{,}  \qquad
h^1_{\theta}(\overline{M})= h^0_{\theta} (\overline{\omega}_{E/\dz}  \otimes \overline{M}^{\vee}) \, \mbox{,}
$$
where $M$ is the corresponding projective $E$-module, $g: \Spec E \to \Spec \dz$  is the natural morphism,
$\parallel \cdot \parallel_{g_* \overline{M}}$ is the Euclidean norm on the $\dr$-vector space $M \otimes_{\dz} \dr$ corresponding to the Hermitian vector bundle $g_* \overline{M}$ over the ``arithmetic curve'' $\Spec \dz$,
 $\overline{M}^{\vee}$ is the dual to $\overline{M}$ Hermitian vector bundle, $\overline{\omega}_{E/\dz}$
 is the canonical Hermitian line bundle over $\Spec E$. Then the Poisson summation formula implies
$$
h^0_{\theta}(\overline{M}) - h^1_{\theta}(\overline{M})=
\widehat{\mathop{\rm deg}} \ \overline{M} - \frac{1}{2} \log \mid \Delta_K \mid \cdot \ {\mathop{\rm rk}} M  \, \mbox{,}
$$
where  $\widehat{\mathop{\rm deg}}$ is the Arakelov degree of a Hermitian bundle, $\Delta_K$  is the discriminant of  the number field $K$.
 }
\end{nt}
\begin{nt}{\em
More on generalizations of Remark~\ref{Weng} and $\theta$-invariants from Remark~\ref{Arakelov} it
will be written in a subsequent paper~\cite{Osip2}.}
\end{nt}

\vspace{0.3cm}

{\small
\noindent Steklov Mathematical Institute of Russsian Academy of Sciences, 8 Gubkina St., Moscow 119991, Russia \\

\smallskip

\noindent National Research University Higher School of Economics, Laboratory of Mirror Symmetry, NRU HSE, 6 Usacheva str., Moscow 119048, Russia \\

\smallskip

\noindent National University of Science and Technology ``MISiS'',  Leninsky Prospekt 4, Moscow  119049, Russia
}

\medskip

\noindent {\it E-mail:}  ${d}_{-} osipov@mi.ras.ru$

\end{document}